\documentclass[11pt]{article}

\usepackage[T1]{fontenc}
\usepackage{lmodern}
\usepackage[a4paper,margin=28mm]{geometry}
\usepackage{amsmath,amssymb,amsthm}
\usepackage{microtype}
\usepackage[hidelinks]{hyperref}
\hypersetup{pdftitle={Rigidity and stability for the Bogovskii constant},
  pdfauthor={Glen Wheeler}}
\newtheorem{theorem}{Theorem}
\newtheorem{proposition}[theorem]{Proposition}
\newtheorem{remark}[theorem]{Remark}
\newcommand{\R}{\mathbb R}
\newcommand{\A}{\mathcal A}
\newcommand{\CB}{C_B}
\DeclareMathOperator{\diver}{div}
\allowdisplaybreaks[1]

\title{Rigidity and stability for the Bogovski\u{\i} constant}
\author{Glen Wheeler\thanks{School of Mathematics and Physics,
University of Wollongong, Australia. Email: \texttt{glenw@uow.edu.au}.}}
\date{12 September 2026}

\begin{document}
\maketitle

\begin{abstract}
We prove that balls are the unique minimisers of the Bogovski\u{\i} constant
among bounded connected Lipschitz domains in $\R^n$, resolving Open Problem~2
of Gazzola, Grunau and Sperone. 
\end{abstract}

\section{Introduction}

The optimal constant for the divergence equation links incompressible flow
with classical functional inequalities. Horgan and Payne~\cite{HP}
related the Babu\v{s}ka--Aziz, Friedrichs and Korn inequalities on smooth
simply connected planar domains; Costabel and Dauge~\cite{CD} established
the identity between the Babu\v{s}ka--Aziz and Friedrichs constants on
arbitrary bounded planar domains. More recently, Gazzola, Grunau and
Sperone~\cite{GGS} gave a unified treatment of the divergence equation and
its optimal constants, proving the universal lower bound $C_B\geq n$.
Together with the known value $C_B(B)=n$, this shows that balls are
minimisers. Their Open Problem~2 asks whether they are the only ones.

Let $\Omega\subset\R^n$, $n\geq2$, be a bounded connected Lipschitz domain,
and let $L^2_0(\Omega)=\{f\in L^2(\Omega):\int_\Omega f=0\}$.
Following~\cite{GGS}, we use the squared-norm convention
\begin{equation}\label{eq:definition}
 \CB(\Omega)=
 \sup_{0\ne f\in L^2_0(\Omega)}\;
 \inf_{\substack{v\in H^1_0(\Omega;\R^n)\\\diver v=f}}
 \frac{\int_\Omega|\nabla v|^2}{\int_\Omega f^2}.
\end{equation}
Here $|\nabla v|$ is the Frobenius norm. The divergence equation has a
bounded right inverse on this class of domains, so $\CB(\Omega)<\infty$.
The constant is invariant under translations and dilations.

Torsion already enters the lower-bound proof in~\cite{GGS}. Our key
observation is to prescribe the centred coordinate functions
$f_i=x_i-\bar x_i$ as divergences, where $\bar x$ is the centroid.
Testing the torsion equation $-\Delta w=1$, $w\in H^1_0(\Omega)$,
against the corresponding component of any lift gives
$I(\Omega)\leq n\CB(\Omega)T(\Omega)$, with
$I(\Omega)=\int_\Omega|x-\bar x|^2$ and $T(\Omega)=\int_\Omega w$.
The identity
$I/n^2-T=\int_\Omega|\nabla w+(x-\bar x)/n|^2$ then shows that
$\CB=n$ forces the torsion function to be quadratic; its zero trace
determines a ball. The same estimate, combined with Saint--Venant and
an elementary comparison of radial moments, gives quantitative stability
in terms of the Fraenkel asymmetry
\[
 \A(\Omega)=\inf_{\substack{B\text{ a ball}\\|B|=|\Omega|}}
 \frac{|\Omega\mathbin\triangle B|}{|\Omega|}.
\]
Our main result is the following.

\begin{theorem}\label{thm:main}
For every bounded connected Lipschitz domain $\Omega\subset\R^n$, $n\geq2$,
one has $\CB(\Omega)\geq n$, with equality if and only if $\Omega$ is a ball.
Moreover,
\begin{align}
 \CB(\Omega)-n
 &\geq n\left[
 \left(1+\frac{\A(\Omega)}2\right)^{1+2/n}
 +\left(1-\frac{\A(\Omega)}2\right)^{1+2/n}-2\right]
 \label{eq:profile}\\
 &\geq \frac{n+2}{2n}\,\A(\Omega)^2.\label{eq:stability}
\end{align}
\end{theorem}

The proof of rigidity below uses no boundary differentiability beyond
Lipschitz regularity. We do not claim that the coefficient
in~\eqref{eq:stability} is optimal.

\section*{Acknowledgements}

The author is grateful to Gazzola, Grunau and Sperone for their interest in
this short note, and for encouraging its dissemination.

\section{A moment--torsion estimate and rigidity}

Let $w\in H^1_0(\Omega)$ solve $-\Delta w=1$ weakly. Write
\[
 T=T(\Omega)=\int_\Omega w=\int_\Omega|\nabla w|^2,
 \qquad \bar x=\frac1{|\Omega|}\int_\Omega x,
 \qquad I=I(\Omega)=\int_\Omega|x-\bar x|^2.
\]

\begin{proposition}\label{prop:moment}
With this notation,
\begin{equation}\label{eq:moment}
 \CB(\Omega)\geq\frac{I}{nT}\geq n,
\end{equation}
and
\begin{equation}\label{eq:torsion-defect}
 \int_\Omega\left|\nabla w+\frac{x-\bar x}{n}\right|^2
 =\frac{I}{n^2}-T
 \leq \frac{\CB(\Omega)-n}{n}\,T.
\end{equation}
\end{proposition}

\begin{proof}
For each $i$, put $f_i=x_i-\bar x_i$ and $m_i=\int_\Omega f_i^2>0$.
Every $v\in H^1_0(\Omega;\R^n)$ satisfying $\diver v=f_i$ obeys
\[
 m_i=\int_\Omega f_i\diver v
 =-\int_\Omega v_i
 =-\int_\Omega\nabla w\cdot\nabla v_i.
\]
Thus $m_i^2\leq T\int_\Omega|\nabla v|^2$, and
\[
 \frac{m_i}{T}\leq
 \inf_{\substack{v\in H^1_0(\Omega;\R^n)\\\diver v=f_i}}
 \frac{\int_\Omega|\nabla v|^2}{m_i}
 \leq\CB(\Omega).
\]
Summing $m_i\leq\CB(\Omega)T$ over $i$ yields $I\leq n\CB(\Omega)T$.

Since $w\in H^1_0(\Omega)$, integration by parts gives
$\int_\Omega(x-\bar x)\cdot\nabla w=-nT$. Expanding the square therefore
gives the identity in~\eqref{eq:torsion-defect}. Its nonnegativity gives
$I\geq n^2T$, and its upper bound follows from $I\leq n\CB(\Omega)T$.
\end{proof}

\begin{proof}[Proof of rigidity in Theorem~\ref{thm:main}]
If $\CB(\Omega)=n$, then~\eqref{eq:torsion-defect} gives
$\nabla w=-(x-\bar x)/n$. By connectedness,
\[
 w(x)=\frac{R^2-|x-\bar x|^2}{2n}\quad\text{in }\Omega
\]
for some $R>0$. Indeed, the strong maximum principle gives $w>0$ in
$\Omega$, hence $\Omega\subset B_R(\bar x)$.
The polynomial representative has zero trace on $\partial\Omega$.
Continuity and the positive surface measure of every Lipschitz boundary
patch imply that it vanishes at every boundary point. Thus
$\partial\Omega\subset\partial B_R(\bar x)$, and $\Omega$ is a nonempty
relatively open and closed subset of $B_R(\bar x)$. It is thus the whole ball.
The converse is the known identity $\CB(B)=n$;
see~\cite[Section~3.1]{GGS}.
\end{proof}

\section{Quantitative stability}

\begin{proof}[Proof of \eqref{eq:profile}--\eqref{eq:stability}]
Let $V=|\Omega|$ and choose $B=B_R(\bar x)$ with $|B|=V$. Direct integration
of $w_B=(R^2-|x-\bar x|^2)/(2n)$ gives
\begin{equation}\label{eq:ball}
 I(B)=\frac{n}{n+2}VR^2,
 \qquad T(B)=\frac{VR^2}{n(n+2)}=\frac{I(B)}{n^2}.
\end{equation}
We use the classical Saint--Venant inequality $T(\Omega)\leq T(B)$.
For completeness, the torsion equation and Cauchy--Schwarz give
\[
 T(\Omega)=\sup_{0\ne\varphi\in H^1_0(\Omega)}
 \frac{\bigl(\int_\Omega\varphi\bigr)^2}
 {\int_\Omega|\nabla\varphi|^2}.
\]
One may take $\varphi\geq0$; Schwarz rearrangement onto $B$ preserves its
integral and does not increase its Dirichlet energy. This proves the
inequality. 

Set $s=|\Omega\setminus B|/V=|B\setminus\Omega|/V\in[0,1]$ and
$p=1+2/n$. Among sets of volume $sV$ outside $B$, the radial moment is
smallest on the shell from radius $R$ to $R(1+s)^{1/n}$. Among sets of
volume $sV$ inside $B$, it is largest on the shell from radius
$R(1-s)^{1/n}$ to $R$. Consequently,
\begin{equation}\label{eq:shells}
 I(\Omega)-I(B)\geq I(B)g(s),
 \qquad g(s)=(1+s)^p+(1-s)^p-2.
\end{equation}
This comparison follows directly by exchanging mass across a sphere:
the weight $|x-\bar x|^2$ increases with the radius. Integrating that
weight over the two shells gives~\eqref{eq:shells}.

By~\eqref{eq:moment}, Saint--Venant and~\eqref{eq:ball}--\eqref{eq:shells},
\[
 \CB(\Omega)\geq\frac{I(\Omega)}{nT(\Omega)}
 \geq n\frac{I(\Omega)}{I(B)}
 \geq n\bigl(1+g(s)\bigr).
\]
Now $2s\geq\A(\Omega)$ and $g$ is increasing on $[0,1]$,
which proves~\eqref{eq:profile}.
Finally $g(0)=g'(0)=0$ and, for $0\leq s<1$,
\[
 g''(s)=p(p-1)\bigl[(1+s)^{p-2}+(1-s)^{p-2}\bigr]
 \geq2p(p-1).
\]
Here $1<p\leq2$, so $t\mapsto t^{p-2}$ is convex on $(0,\infty)$.
Twice integrating, and using continuity at $s=1$, gives
$g(s)\geq p(p-1)s^2$. Since $np(p-1)/4=(n+2)/(2n)$,
\eqref{eq:stability} follows.
\end{proof}

\begin{remark}[Related inequalities]
With the Dirichlet-energy dual norm
\[
 \|\nabla f\|_{-1}
 =\sup_{0\ne v\in H^1_0(\Omega;\R^n)}
 \frac{|\int_\Omega f\diver v|}{\|\nabla v\|_2},
\]
the best constant in the Ne\v{c}as inequality
$\|f\|_2^2\leq C\|\nabla f\|_{-1}^2$, $f\in L^2_0(\Omega)$,
is $C=\CB(\Omega)$~\cite[Proposition~2.11]{GGS}.
Theorem~\ref{thm:main} therefore gives the same rigidity and stability
for this constant. In dimension two, let $\Gamma(\Omega)$ be the best
constant in $\int_\Omega h^2\leq\Gamma(\Omega)\int_\Omega k^2$ for
$h+ik$ holomorphic with $h,k\in L^2(\Omega)$ and
$\int_\Omega h=\int_\Omega k=0$.
The identity $\CB=1+\Gamma$ of Costabel and
Dauge~\cite[Theorem~2.1]{CD} gives
\[
 \Gamma(\Omega)-1\geq\A(\Omega)^2,
 \qquad \Gamma(\Omega)=1\ \Longleftrightarrow\ \Omega\text{ is a disk}.
\]
\end{remark}

\begin{remark}[Attainment on smooth domains]
The condition $\CB(\Omega)>2$ can be removed from the attainment assertion
of~\cite[Theorem~2.15]{GGS}, retaining its $C^4$ boundary hypothesis.
The only additional case is $n=2$, $\CB=2$, hence a disk. There the fields
$w e_i$ minimise energy for their prescribed divergences and have quotient
$2$: they are energy-orthogonal to every zero-divergence field, since
$\int_\Omega v_i=-\int_\Omega x_i\diver v=0$ for such fields.
\end{remark}

\end{document}